\documentclass[12pt]{article}

\usepackage[T1]{fontenc}
\usepackage[active]{srcltx}
\usepackage{lmodern}
\usepackage{amsmath,amstext,amssymb,graphicx,graphics,color}
\usepackage{tikz}
\usepackage{mleftright}
\usepackage{pgfplots}
\usepackage{theorem}
\usepackage{cite}               % to cite in increasing order
\usepackage{hyperref}           % faire un pdf avec des liens dynamiques
\usepackage{bbm}                % \mathbbm{1}
\usepackage{fancybox}           % pour \ovalbox
\usepackage[section]{placeins}  % \FloatBarrier
\usepackage{subcaption}
\usepackage{enumerate}
\theorembodyfont{\normalfont\slshape} % on remet la config normal
\newtheorem{definition}{Definition}
\newtheorem{theorem}{Theorem}
\newtheorem{proposition}{Proposition}[section]
\newtheorem{corollary}[proposition]{Corollary}

\newtheorem{lemma}[proposition]{Lemma}
\theoremstyle{break} % les theorem vont à la ligne dorénavant ("plain" otherwise)

\newenvironment{proof}%
{{\par\noindent \bf Proof. \nobreak}}%
{\nobreak \removelastskip \nobreak \hfill $\Box$ \medbreak}
{{\par\noindent \bf Proof \nobreak}}%
{\nobreak \removelastskip \nobreak \hfill $\Box$ \medbreak}
{{\par\noindent \bf Proof lemma. \nobreak}}%
{\nobreak \removelastskip \nobreak \bf End proof lemma. \medbreak}
\newenvironment{remark}{\par \medskip \noindent {\bf Remark. }\nobreak}{\par \medskip}

\usepackage{fancyhdr}
 \fancypagestyle{plain}{
  \fancyhead{}
  
}
\def\paragraph#1{{\bf #1\ }}
\newcommand{\RN}[1]{%
  \textup{\uppercase\expandafter{\romannumeral#1}}%
}
\newcommand{\expo}{\mathrm{e}}

\newcommand{\dd}{\mathrm{d}}

\newcommand{\NN}{\mathrm{N}}

\newcommand{\overbar}[1]{\mkern 1.5mu\overline{\mkern-1.5mu#1\mkern-1.5mu}\mkern 1.5mu}
\usepackage[utf8]{inputenc}
\usepackage{unicode_character}

\title{A biased dollar exchange model involving bank and debt with discontinuous equilibrium}

\author{Fei Cao \footnotemark[1] \and Stephanie Reed \footnotemark[2]}

\begin{document}
\maketitle

\footnotetext[1]{University of Massachusetts Amherst - Department of Mathematics and Statistics, Amherst, MA 01003, USA}
\footnotetext[2]{California State University, Fullerton - Department of Mathematics, 800 N State College Blvd, Fullerton, CA 92831, USA}

\tableofcontents

\begin{abstract}
In this work, we investigate a biased dollar exchange model with collective debt limit, in which agents picked at random (with a rate depending on the amount of dollars they have) give at random time a dollar to another agent being picked uniformly at random, as long as they have at least one dollar in their pockets or they can borrow a dollar from a central bank if the bank is not empty. This dynamics enjoys a mean-field type interaction and partially extends the recent work \cite{cao_uncovering_2022} on a related model. We perform a formal mean-field analysis as the number of agents grows to infinity and as a by-product we discover a two-phase (ODE) dynamics behind the underlying stochastic $N$-agents dynamics. Numerical experiments on the two-phase (ODE) dynamics are also conducted where we observe the convergence towards its unique equilibrium in the large time limit.
\end{abstract}

\noindent {\bf Key words: Econophysics, Agent-based model, Interacting agents, Mean-field, Two-phase, Bank}

\section{Introduction}
\setcounter{equation}{0}

Econophysics is a subfield of statistical physics that draw concepts and techniques from traditional physics and apply them to economics, finance, and related fields, and we refer the readers to the pioneer work \cite{dragulescu_statistical_2000} for many models motivated from econophysics. Among many important tasks in this area of research, a fundamental mission is to demystify how various (macroscopic) economical phenomena could be explained by (microscopic) laws in classical statistical physics under certain modeling assumptions, and we refer to \cite{kutner_econophysics_2019,pareschi_interacting_2013,pereira_econophysics_2017,savoiu_econophysics_2013} for a general review on using kinetic theory for the modelling of wealth (re-)distributions.

Motivations for studying econophysics models are at least two-fold: from the perspective of a policy maker, how to exert influence on the growing wealth inequality in order to mitigate the alarming gap between rich and poor is a central issue to be dealt with. From a mathematical viewpoint, the underlying mechanisms behind the formation of macroscopic phenomena, for instance various possible wealth distributions arising from distinct agent-based dollar exchange models, remain to be thoroughly understood. For a given agent-based model, we aim at identifying a limit dynamics, which is fully deterministic, when we send the number of individuals/players to infinity, and then the deterministic system will be further analyzed with the hope of proving its convergence to equilibrium (if there is one) in the large time limit. This philosophy has been successfully implemented among the overabundant literatures across different fields of applied mathematics, see for instance \cite{cao_asymptotic_2020,cao_k_2021,carlen_kinetic_2013,motsch_short_2018}.

In this work, we consider a simple mechanism for money exchange involving a bank, meaning that there are a fixed number of agents (denoted by $N$) and one bank. We denote by $S_i(t)$ the amount of dollars the agent $i$ has at time $t$ and we suppose that $\sum_{i=1}^N S_i(0) = N\,\mu$ for some fixed $\mu \in \mathbb{N}_+$. Thus, each agent in this closed economic system has $\mu$ dollars on average. Moreover, we denote by $B_* := B_c + B_d$ the initial amount of dollars in the bank, where $B_c$ and $B_d$ represent the amount of dollars owned by the bank in the form of ``cash'' and in the form of ``debt'' (borrowed by agents), respectively. Also, we introduce another parameter $\nu \in \mathbb{N}_+$, which measures the ratio of the bank’s initial wealth to the initial combined wealth of all the agents, and set $B_* = N\,\mu\,\nu$ to be the total amount of dollars put in the bank initially. We emphasize here that the parameters $\mu$ and $\nu$ are both dimensionless constants.

The model investigated in this work was a natural variant of the model proposed in \cite{xi_required_2005} and revisited in several recent works \cite{cao_uncovering_2022,lanchier_rigorous_2018-1}, which we now briefly recall first: at random times (generated by an exponential law), an agent $i$ (the ``giver'') and an agent $j$ (the ``receiver'') are picked uniformly at random. If either the ``giver'' $i$ has at least one dollar (i.e. $S_i\geq 1$) or if the central bank has ``cash'' (i.e. $B_c\geq 1$), then the receiver $j$ receives a dollar. Otherwise, when the giver $i$ has no dollar and the bank has no cash, then nothing happens. To further clarify the aforementioned model, we emphasize that dollars in the central bank are untouched when the giver $i$ has at least one dollar, and the giver $i$ will borrow a dollar from the bank (as long as the bank has money) if he/she has no dollar to give. The aforementioned model is termed as the \textbf{unbiased exchange model with collective debt limit} in \cite{cao_uncovering_2022} and the \textbf{one-coin model with collective debt limit} in \cite{lanchier_rigorous_2018-1}. Notice that when the bank gives a dollar to agent $j$, there is still one dollar withdrew from the giver $i$, i.e., the debt of agent $i$ increases. The debt of agent $i$ could be reduced once it will become a ``receiver''. It is also important to notice that in this model the bank never loses money, it just transforms its ``cash'' $B_c$ into ``debt'' (and vice versa).  Without the bank, agents can only give a dollar when they have at least one dollar (hence no debt is allowed).

We now introduce the model considered in the present manuscript as a biased version of the aforementioned model, i.e., we drop the rule of selecting the ``giver'' uniformly at random in favor of the requirement that the ``giver'' agent $i$ will be picked (at each jump time) at a rate proportional to $f(S_i)$, where $f \colon \mathbb R \to (0,\infty)$ is a strictly positive function. From now on, we terms the model as the \textbf{$f$-biased exchange model with collective debt limit}, which can be represented by \eqref{biased_exchange_with_debt} below.
\begin{equation}
\label{biased_exchange_with_debt}
(S_i,S_j)~ \begin{tikzpicture} \draw [->,decorate,decoration={snake,amplitude=.4mm,segment length=2mm,post length=1mm}]
  (0,0) -- (.6,0); \node[above,red] at (0.3,0) {\small{$f(S_i)$}};\end{tikzpicture}~  (S_i-1,S_j+1) \qquad \text{ if}~ S_i\geq 1~ \text{{\bf or}}~ B_c \geq 1.
\end{equation}
We emphasize that the unbiased exchange model with collective debt limit investigated in \cite{cao_uncovering_2022} is a special case of the model we introduced here where $f$ is a fixed positive constant (which can be normalized to $1$ without loss of generality). For now the function $f$ is an arbitrary positive function from $\mathbb R$ to $(0,\infty)$, but later we will impose appropriate constraints on $f$ for reasons to be further clarified in the subsequent development of our analysis.

\begin{remark}
In order to have the correct asymptotic as the number of agents goes to infinity $N \to \infty$, we need to adjust the rate $f(S_i)$ by normalizing by $N$, which amounts to replacing $f(S_i)$ in \eqref{biased_exchange_with_debt} by $f(S_i) \slash N$ so that the rate of a typical agent giving a dollar per unit time is of order $1$. Here by saying ``a typical agent'' we mean an agent picked (uniformly) at random from the crowd of $N$ agents.
\end{remark}

The fundamental problem of interest is the derivation/idenfication of the limiting money distribution among the agents as the total number of agents $N$ and time $t$ become large. We propose two analytical approaches for addressing this problem with illustration provided by figure \ref{fig:scheme_sketch_full} below.  The first approach, detailed in section \ref{sec:sec2}, relies on the study of the limiting money distribution for all (fixed) values of the number $N$ of agents as time goes to infinity and then simplification of the probability that a typical agent has $n$ dollars at (``thermodynamic'') equilibrium in the large population limit.

\begin{figure}[!htb]
  \centering
  \includegraphics[width=.97\textwidth]{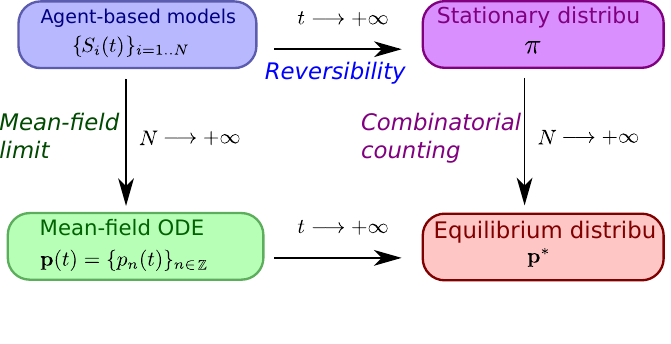}
\caption{Schematic illustration of the strategy of the limiting procedures.}
\label{fig:scheme_sketch_full}
\end{figure}

The second approach is a mean-field approach carried out in the spirit of the recent work \cite{cao_uncovering_2022}, which will be presented in section \ref{sec:sec3}. The first step in the mean-field approach lies in the derivation of its limit dynamics as the number of agents goes to infinity (i.e., $N \to \infty$). With this goal, we introduce the probability distribution of dollars:
\begin{equation}
  \label{eq:p}
  {\bf p}(t)=\left(\ldots,p_{-n}(t),\ldots,p_{-1}(t),p_0(t),p_1(t),\ldots,p_n(t),\ldots\right) %\quad \text{for } n ∈ℤ
\end{equation}
with $p_n(t)= \{``\text{probability that a typical agent has } n \text{ dollars at time}~ t "\}$. The evolution of ${\bf p}(t)$ will be given by a system of (deterministic) nonlinear ordinary differential equations. However, to rigorously justify the transition from a stochastic interacting agents systems into a deterministic set of ODEs, one needs the so-called \emph{propagation of chaos} \cite{sznitman_topics_1991}. Heuristically speaking, the propagation of chaos property in the specific context of our model says that the (random) agent $i$ and agent $j$ picked to engage in the binary trading activity \eqref{biased_exchange_with_debt} are \emph{statistically independent} as $N \to \infty$. The rigorous derivation of the mean-field ODE system \eqref{eq:evolution_p} from its underlying stochastic $N$-agents dynamics is out of the scope of the present paper, but the derivation has been fully justified in various models arising from econophysics, see for instance \cite{cao_derivation_2021,cao_entropy_2021,cao_explicit_2021,cao_interacting_2022,cortez_uniform_2022}. The major difficulty in the mean-field argument lies in the fact that the evolution of ${\bf p}(t)$ is split into two phases. Indeed, the evolution of ${\bf p}(t)$ changes at the first time when there is no more cash in the bank. We will denote by $t_*$ the time at which such event occur, i.e., at $t_*$ the bank is empty for the first time. We will show that the evolution equation for ${\bf p}(t)$ takes the following form
\begin{equation}
  \label{eq:evolution_p}
  \begin{array}{cccc}
  \text{\bf Phase I:} &  \hspace{.5cm}  \partial_t {\bf p} = Q_1[{\bf p}] \hspace{.5cm} & \text{for} & 0≤t≤t_* \\ \smallskip
  \text{\bf Phase II:} & \hspace{.5cm} \partial_t {\bf p} =  Q_2[{\bf p}] \hspace{.5cm} & \text{for} & t>t_*
  \end{array}
\end{equation}
where the exact expressions for the operator $Q_1$ and $Q_2$ are given by \eqref{eq:Q_1_b} and \eqref{eq:Q_2_b}, respectively. Once the large population limit has been carried out and a limit (deterministic) dynamics has been uncovered, one naturally attempts to investigate the asymptotic behavior of the probability mass function ${\bf p}(t)$ as $t \to \infty$, with the intent of proving convergence to the unique equilibrium distribution (if such unique equilibrium exists). Unfortunately, we do not aim to address the problem of the convergence to equilibrium for the two-phase ODE system \eqref{eq:evolution_p} analytically, and we will be content with numerical evidence on such convergence behavior.

Although we will only investigate a specific binary dollar exchange model in the current work, other exchange rules can also be imposed and studied, leading to other econophysics models. To name a few, the so-called immediate exchange model introduced in \cite{heinsalu_kinetic_2014} assumes that pairs of agents are randomly and uniformly picked at each random time, and each agent transfers a random fraction of their money to the other agent, where these fractions are independent and uniformly distributed on $[0,1]$. The so-called uniform reshuffling model investigated in \cite{cao_entropy_2021,dragulescu_statistical_2000,lanchier_rigorous_2018,matthes_steady_2008} proposes that the total amount of money of two randomly and uniformly picked agents possessed before interaction is uniformly redistributed among the two agents after interaction. The so-called repeated averaging model studied for instance in \cite{cao_explicit_2021} studies the mechanism where two randomly selected agents share half of their wealth with each other at each binary exchange. The binomial reshuffling model proposed in a recent work \cite{cao_binomial_2022} is a variant of the uniform reshuffling mechanism in which the agents' combined wealth is redistributed according to a binomial distribution. For models with saving propensity, models involving debts, or other econophysics models, we refer the readers to \cite{boghosian_h_2015,chakraborti_statistical_2000, chatterjee_pareto_2004, lanchier_rigorous_2018-1,cao_uncovering_2022,cohen_bounding_2023,during_kinetic_2008,torregrossa_wealth_2017} and the references therein.

\section{Markov chain approach}\label{sec:sec2}
\setcounter{equation}{0}

In this section, we will tackle the problem of finding the limiting money distribution (as $N,\,t \to \infty$) using the techniques employed in \cite{lanchier_rigorous_2017,lanchier_rigorous_2018,lanchier_rigorous_2018-1,lanchier_role_2018}. Define $\mathcal{C}(N,M,B_*)$ to be the set of possible configurations with $N$ agents having a total initial wealth of $M = N\,\mu$ dollars, where the initial amount of dollars in the bank equals $B_* = N\,\mu\,\nu$, i.e.,
\[\mathcal{C}(N,M,B_*) = \{S \colon \{1,2,\ldots,N\}\to \mathbb{Z} \mid \sum_{i=1}^N S_i = M,\,\, -B_* \leq S_i \leq M + B_*\}.\]

Let us fix any $\xi\in \mathcal{C}(N-1,M-1,B_*)$ and introduce $\xi^i\in \mathcal{C}(N,M,B_*)$ defined via
%\[\xi^i_j = \xi_j + \mathbbm{1}\{i=j\}.\]
$$\xi^i(j) = \xi(i) + \mathbbm{1}\{i=j\}.$$

Then $\{S(t)\}_{t\geq 0}$ is a continuous-time Markov chain having state space $\mathcal{C}(N,M,B_*)$. Note also that all transitions of the process are of the form $\xi^i\to \xi^j$ for some $i\neq j$.

As indicated in the introduction, the rates of transition from $\xi^i$ to $\xi^j$ with $i\neq j$ for the process $S(t)$ are given by
\[q(\xi^i,\xi^j) = \frac{f(\xi^i(i))}{N}.\]
We first show that $\{S(t)\}_{t\geq 0}$ is a finite, irreducible and aperiodic Markov chain, whence it admits a unique stationary distribution.

\begin{lemma}\label{lem:usd}
The process $\{S(t)\}_{t\geq 0}$ has a unique stationary distribution $\pi:\mathcal{C}(N,M,B_*) \to \mathbb{R}_+$ such that
\[\lim_{t\to\infty}\mathbb P[S(t) = \xi] = \pi(\xi).\]
\end{lemma}

\begin{proof}
The proof of this result can be carried out in a similar fashion as shown in earlier works \cite{lanchier_rigorous_2017,lanchier_rigorous_2018,lanchier_rigorous_2018-1,lanchier_role_2018} for other econophysics models, hence we will skip the details and leave it to the readers.
\end{proof}

\begin{lemma}\label{lem:rev}
The process $\{S(t)\}_{t\geq 0}$ is reversible and has the stationary distribution given by
\begin{equation}\label{eq:stationary_dist}
\pi(\xi) = \frac{\theta(\xi)}{\sum_{\xi'\in\mathcal{C}(N,M,B_*)} \theta(\xi')},
\end{equation}
where
\begin{equation}\label{eq:def_of_theta}
\displaystyle\theta(\xi') = \prod_{i=1}^N\prod_{j=-B_*}^{\xi'(i)} \frac{1}{f(j)}.
\end{equation}
\end{lemma}

\begin{proof}
Let $\xi^i,\xi^j\in\mathcal{C}(N,M,B_*)$ with $i\neq j$. Then
%\[q(\xi^i,\xi^j) = \frac{f(\xi^i_i)}{N} = \frac{f(\xi_i+1)}{N}\]
\[q(\xi^i,\xi^j) = \frac{f(\xi^i(i))}{N} = \frac{f(\xi(i)+1)}{N},\]
and
%\[q(\xi^j,\xi^i) = \frac{f(\xi^j_j)}{N} = \frac{f(\xi_j+1)}{N},\]
\[q(\xi^j,\xi^i) = \frac{f(\xi^j(j))}{N} = \frac{f(\xi(j)+1)}{N},\]
which yields
\begin{equation}\label{eq:rev1}
%\frac{q(\xi^i,\xi^j)}{f(\xi_i+1)} = \frac{q(\xi^j,\xi^i)}{f(\xi_j+1)}.
\frac{q(\xi^i,\xi^j)}{f(\xi(i)+1)} = \frac{q(\xi^j,\xi^i)}{f(\xi(j)+1)}.
\end{equation}
Note that for any $k=1,2,\ldots,N$, we have that
\begin{equation}\label{eq:rev2}
\begin{split}
%\theta(\xi^k) & = \prod_{i=1}^N\prod_{j=-B_*}^{\xi^k_i} \frac{1}{f(j)}\\
\theta(\xi^k) & = \prod_{i=1}^N\prod_{j=-B_*}^{\xi^k(i)} \frac{1}{f(j)}\\
% & = \left(\prod_{\substack{i=1 \\ i\neq k}}^N\prod_{j=-B_*}^{\xi^k_i} \frac{1}{f(j)}\right)\left(\prod_{j=-B_*}^{\xi^k_k} \frac{1}{f(j)}\right)\\
 & = \left(\prod_{\substack{i=1 \\ i\neq k}}^N\prod_{j=-B_*}^{\xi^k(i)} \frac{1}{f(j)}\right)\left( \prod_{j=-B_*}^{\xi^k(k)} \frac{1}{f(j)}\right)\\
%& = \left(\prod_{\substack{i=1 \\ i\neq k}}^N\prod_{j=-B_*}^{\xi_i} \frac{1}{f(j)}\right)\left( \prod_{j=-B_*}^{\xi_k+1} \frac{1}{f(j)}\right)\\
& = \left(\prod_{\substack{i=1 \\ i\neq k}}^N\prod_{j=-B_*}^{\xi(i)} \frac{1}{f(j)}\right)\left( \prod_{j=-B_*}^{\xi(k)+1} \frac{1}{f(j)}\right)\\
%& = \left(\prod_{\substack{i=1 \\ i\neq k}}^N\prod_{j=-B_*}^{\xi_i} \frac{1}{f(j)}\right)\left( \prod_{j=-B_*}^{\xi_k} \frac{1}{f(j)}\right)\frac{1}{f(\xi_k+1)}\\
& = \left(\prod_{\substack{i=1 \\ i\neq k}}^N\prod_{j=-B_*}^{\xi(i)} \frac{1}{f(j)}\right)\left( \prod_{j=-B_*}^{\xi(k)} \frac{1}{f(j)}\right)\frac{1}{f(\xi(k)+1)}\\
%& = \left(\prod_{i=1}^N\prod_{j=-B_*}^{\xi_i} \frac{1}{f(j)}\right)\frac{1}{f(\xi_k+1)} = \frac{\theta(\xi)}{f(\xi_k+1)}.
& = \left(\prod_{i=1}^N\prod_{j=-B_*}^{\xi(i)} \frac{1}{f(j)}\right)\frac{1}{f(\xi(k)+1)} = \frac{\theta(\xi)}{f(\xi(k)+1)}.
 \end{split}
\end{equation}
Putting equations \eqref{eq:rev1} and \eqref{eq:rev2} together gives
\begin{equation}
%\theta(\xi^i)q(\xi^i,\xi^j) = \frac{\theta(\xi)}{f(\xi_i+1)}q(\xi^i,\xi^j)=\frac{\theta(\xi)}{f(\xi_j+1)}q(\xi^j,\xi^i) = \theta(\xi^j)q(\xi^j,\xi^i),
\theta(\xi^i)q(\xi^i,\xi^j) = \frac{\theta(\xi)}{f(\xi(i)+1)}q(\xi^i,\xi^j)=\frac{\theta(\xi)}{f(\xi(j)+1)}q(\xi^j,\xi^i) = \theta(\xi^j)q(\xi^j,\xi^i),
\end{equation}
which ends the proof.
\end{proof}

We are now ready to state the following general convergence result, which follows readily from Lemmas \ref{lem:usd} and \ref{lem:rev}.
\begin{theorem}\label{thm:genf_dist}
The fraction of agents with $n$ dollars in the long run, or equivalently the probability that a typical agent has $n$ dollars in the large time limit, is given by
\begin{equation}\label{eq:convergence}
%\lim_{t\to\infty} \mathbb P[S_i(t) = n] = \displaystyle\frac{\displaystyle\sum_{\xi\,:\,\xi_i = n} \pi(\xi)}{\displaystyle\sum_{\xi} \pi(\xi)} = \displaystyle\frac{\displaystyle\sum_{\xi\,:\,\xi_i = n} \theta(\xi)}{\displaystyle\sum_{\xi} \theta(\xi)}.
\lim_{t\to\infty} \mathbb P[S_i(t) = n] = \displaystyle\displaystyle\sum_{\xi\,:\,\xi(i) = n} \pi(\xi)= \displaystyle\frac{\displaystyle\sum_{\xi\,:\,\xi(i) = n} \theta(\xi)}{\displaystyle\sum_{\xi} \theta(\xi)}.
\end{equation}
\end{theorem}

From now on, we will take a ``carefully crafted'' nonlinear function $f$, which is provided in the following definition.
\begin{definition}\label{def_specific_f}
We define $f_* \colon \mathbb R \to (0,\infty)$ to be $f_*(x) \equiv 1$ for all $x \leq 1$ and $f_*(x) = \frac{x-1}{x}$ for $x > 1$.
\end{definition}

The motivation to consider and treat this particular (``strange looking'') nonlinear $f$ is at least two-fold and will be elaborated in section \ref{sec:sec3} when we employ a mean-field approach. For the moment, we proceed with this specific choice of $f$ to deduce the following corollary.

\begin{corollary}\label{corollary:theta}
If $f$ is taken to be $f_*$ as in Definition \ref{def_specific_f}, then %$\theta(\xi) = \prod\limits_{i\,:\,\xi_i>0} \xi(i)$.
$\theta(\xi) = \prod\limits_{i\,:\,\xi(i)>0} \xi(i)$.
\end{corollary}

\begin{proof}
Let $\xi\in\mathcal{C}(N,M,B_*)$. Then for $i=1,2,\ldots,N$ we have
\begin{equation}\label{eq:c1}
%\prod_{j=-B_*}^{\xi_i} \frac{1}{f_*(j)}=
%    \begin{cases}
%        1 & \text{if } \xi_i\leq 0,\\
%        \xi_i & \text{if } \xi_i > 0,
%    \end{cases}
\prod_{j=-B_*}^{\xi(i)} \frac{1}{f_*(j)}=
    \begin{cases}
        1 & \text{if } \xi(i)\leq 0,\\
        \xi(i) & \text{if } \xi(i) > 0,
    \end{cases}
\end{equation}
as $f_*(n) = 1$ when $n \leq 0$ and $\prod\limits_{j=-B_*}^{n} f_*(j) = 1\cdot \frac{1}{2}\cdot \frac{2}{3}\cdots \frac{n-1}{n} = \frac{1}{n}$ for $n> 0$. Inserting \eqref{eq:c1} into the definition \eqref{eq:def_of_theta} of $\theta(\xi)$ yields the advertised result.
\end{proof}

We now state some elementary identities from combinatorics, whose proofs are standard.
\begin{lemma}\label{lem:comb1}
For each $a \in \mathbb N$ and $b \in \mathbb N_+$, we have
\begin{equation}
\sum_{\substack{y_1+y_2\cdots+y_b = a \\ y_i\ge 0}}1 = \binom{a+b-1}{b-1}.
\end{equation}
\end{lemma}

\begin{proof}
$\displaystyle\sum_{\substack{y_1+y_2\cdots+y_b = a \\ y_i\ge 0}} 1$ is the number of integer solutions to $y_1+y_2\cdots+y_b = a$ with $y_i\ge 0$, which is known to be $\displaystyle\binom{a+b-1}{b-1}$.
\end{proof}

\begin{lemma}\label{lem:comb2} We have for each $n \in \mathbb N$ and $r \in \mathbb N_+$ that
\begin{equation}
\sum_{i=1}^n \,i\,\binom{n-i+r-1}{2r-1} = \binom{n+r}{2r+1} \,\,\,\text{and}\,\,\, \sum_{i=1}^n \,i\,\binom{n-i+r-1}{2r} = \binom{n+r}{2r+2}.
\end{equation}
\end{lemma}

\begin{proof}
%This result can easily be shown by induction on the quantity $r+n$ together with the utilization of Pascal's identity.
For ease of writing, let us define
\begin{equation}\label{eqn:s}
s(n,r) = \sum_{i=1}^n \,i\,\binom{n-i+r-1}{2r-1},
\end{equation}
and
\begin{equation}\label{eqn:u}
u(n,r) =  \sum_{i=1}^n \,i\,\binom{n-i+r-1}{2r}.
\end{equation}
We shall prove the desired result with induction on $n+r$.\\\\
Base Case: When $n=0$ or $r=1$, then we have the following results for $s(0,r)$, $s(n,1)$, $u(0,r)$ and $u(n,1)$.

%\begin{equation*}
%\begin{split}
%s(0,r) &= \sum_{i=1}^0 \,i\,\binom{0-i+r-1}{2r-1}\\
%&= 0\\
%&= \binom{0+r}{2r+1},
%\end{split}
%\end{equation*}

%\begin{equation*}
%\begin{split}
%s(n,1) & = \sum_{i=1}^n \,i\,\binom{n-i+1-1}{2\cdot 1-1}\\
% & = \sum_{i=1}^n \,i\,\binom{n-i}{1}\\
% & = \sum_{i=1}^n \,i\,(n-i)\\
% & = \frac{1}{6}\,(n+1)\,n\,(n-1)\\
% & = \binom{n+1}{3}\\
% & = \binom{n+1}{2\cdot 1+1},\\
% \end{split}
%\end{equation*}

\begin{equation*}
s(0,r) = \binom{0+r}{2r+1}, \,\,\, s(n,1) = \binom{n+1}{2\cdot 1+1},
\end{equation*}

%\begin{equation*}
%\begin{split}
%u(0,r) &= \sum_{i=1}^0 \,i\,\binom{0-i+r-1}{2r}\\
%&= 0\\
%&= \binom{0+r}{2r+2},
%\end{split}
%\end{equation*}

and

%\begin{equation*}
%\begin{split}
%u(n,1) & = \sum_{i=1}^n \,i\,\binom{n-i+1-1}{2\cdot 1}\\
% & = \sum_{i=1}^n \,i\,\binom{n-i}{2}\\
% & = \sum_{i=1}^n \,\frac{1}{2}\,i\,(n-i)\,(n-i-1)\\
% &= \frac{1}{24}\,(n+1)\,n\,(n-1)\,(n-2)\\
% &= \binom{n+1}{4}\\
% &= \binom{n+1}{2\cdot 1+2}.
% \end{split}
%\end{equation*}

\begin{equation*}
u(0,r) =  \binom{0+r}{2r+2}, \,\,\, u(n,1)  = \binom{n+1}{2\cdot 1+2}.
\end{equation*}

Inductive Step: Assume that our hypothesis holds for all $n+r$ such that $n \in \mathbb N$, $r \in \mathbb N_+$ and $1\leq n+r < m$ for some $m>1$. Then setting $n+r = m$ and employing Pascal's identity we have

\begin{equation*}
\begin{split}
s(n,r) &= \sum_{i=1}^n \,i\,\binom{n-i+r-1}{2r-1}\\
&= \sum_{i=1}^n \,i\,\binom{(n-1)-i+r-1}{2r-1}+\sum_{i=1}^n \,i\,\binom{n-i+(r-1)-1}{2r-2}\\
&= \sum_{i=1}^{n-1} \,i\,\binom{(n-1)-i+r-1}{2r-1}+\sum_{i=1}^n \,i\,\binom{n-i+(r-1)-1}{2r-2}\\
&= s(n-1,r) + u(n,r-1)\\
&=\binom{n-1+r}{2r+1}+\binom{n+r-1}{2r}= \binom{n+r}{2r+1},
\end{split}
\end{equation*}

and

\begin{equation*}
\begin{split}
u(n,r) &= \sum_{i=1}^n \,i\,\binom{n-i+r-1}{2r}\\
&= \sum_{i=1}^n \,i\,\binom{(n-1)-i+r-1}{2r}+\sum_{i=1}^n \,i\,\binom{(n-1)-i+r-1}{2r-1}\\
&= \sum_{i=1}^{n-1} \,i\,\binom{(n-1)-i+r-1}{2r}+\sum_{i=1}^{n-1} \,i\,\binom{(n-1)-i+r-1}{2r-1}\\
&= u(n-1,r) + s(n-1,r)\\
&=\binom{n+r-1}{2r+2}+\binom{n+r-1}{2r+1}= \binom{n+r}{2r+2}.
\end{split}
\end{equation*}

\end{proof}

\begin{lemma}\label{lem:comb3}
For each $n \in \mathbb N$ and $r \in \mathbb N_+$,
\begin{equation}
\sum_{\substack{x_1+x_2\cdots+x_r = n \\ x_i\ge 0}}x_1\,x_2\,\cdots\, x_r = \binom{n+r-1}{2\,r-1}.
\end{equation}
\end{lemma}

\begin{proof}
%The result can easily be shown by induction on $r$ using Lemma~\ref{lem:comb2}.
We will show the desired result using induction on $r$.\\
Base Case: $r=1$. We have

\begin{equation*}
\sum_{\substack{x_1+x_2\cdots+x_r = n \\ x_i\ge 0}}x_1\,x_2\,\cdots\, x_r
= n
= \binom{n+1-1}{2\cdot 1 - 1}.
\end{equation*}

Inductive Step: Assume that our hypothesis holds for $r=m$ where $m>1$. Then setting $r=m+1$ and employing Lemma~\ref{lem:comb2}, we have that
\begin{equation*}
\begin{split}
\sum_{\substack{x_1+x_2\cdots+x_r = n \\ x_i\ge 0}}x_1\,x_2\,\cdots\, x_r &= \sum_{\substack{x_1+x_2\cdots+x_{m+1} = n \\ x_i\ge 0}}x_1\,x_2\,\cdots\, x_{m+1}\\
&= \sum_{i=1}^n i\,\sum_{\substack{x_1+x_2\cdots+x_{m} = n-i \\ x_i\ge 0}}x_1\,x_2\,\cdots\, x_{m}\\
&= \sum_{i=1}^n i\,\binom{n+m-i-1}{2m-1}\\
&= \binom{n+m}{2m+1}= \binom{n+r-1}{2r-1}.\\
\end{split}
\end{equation*}
\end{proof}

Now we can specialize the content of Theorem \ref{thm:genf_dist} to deduce the following convergence result.
\begin{corollary}\label{corollary:convegence_f*}
If $f$ is taken to be $f_*$ as in Definition \ref{def_specific_f}, set
\begin{equation}
\varphi(N,M,B_*) = \displaystyle\sum_{\xi}\theta(\xi)
\end{equation}
then
\[\varphi(N,M,B_*) = \sum_{a=0}^{B_*} \sum_{b=0}^{N-1}\binom{N}{b}\binom{a+b-1}{b-1}\binom{M+a+N-b-1}{2(N-b)-1}.\]
Consequently, the distribution of dollars among the $N$-agents system in the large time limit is given by
\begin{equation}\label{eq:converg_f*}
\lim_{t\to\infty} \mathbb P[S_i(t) = n] =
    \begin{cases}
        \frac{n\, \varphi(N-1,M-n,B_*)}{\varphi(N,M,B_*)} & \text{if } n > 0,\\
        \frac{\varphi(N-1,M-n,B_* + n)}{\varphi(N,M,B_*)} & \text{if } n\leq 0.
    \end{cases}
\end{equation}
\end{corollary}

\begin{proof}
If we define $\phi(a,b,N,M,B_*)$ to be the weighted sum of $\theta(\xi)$ over all configurations with $N$ agents, $M$ dollars of agents combined initial wealth, $B_*$ dollars in the bank initially, and with $a$ dollars of total debt and $b$ agents with less than or equal to zero dollars, then we have that
\begin{equation} \label{eq1}
\varphi(N,M,B_*) = \displaystyle\sum_{\xi}\theta(\xi) = \sum_{a=0}^{B_*} \sum_{b=0}^{N-1} \phi(a,b,N,M,B_*).
\end{equation}

Invoking Lemmas \ref{lem:comb1} and \ref{lem:comb3} gives rise to
\begin{equation} \label{eq1}
\begin{split}
\phi(a,b,N,M,B_*) & = \binom{N}{b}\sum_{\substack{y_1+y_2+\cdots+y_b = a \\ y_i\ge 0}} \,\,\,\sum_{\substack{x_1+x_2+\cdots+x_{N-b} = M+a\\ x_i > 0}} x_1\,x_2\,\cdots\, x_{N-b} \\
 & = \binom{N}{b}\binom{a+b-1}{b-1} \,\,\,\sum_{\substack{x_1+x_2+\cdots+x_{N-b} = M+a\\ x_i > 0}} x_1\,x_2\,\cdots\, x_{N-b} \\
  & = \binom{N}{b}\binom{a+b-1}{b-1}\binom{M+a+N-b-1}{2(N-b)-1},
\end{split}
\end{equation}
whence the proof is completed.
\end{proof}

\begin{remark}
Even though the equilibrium distribution (for any fixed $N \in \mathbb N_+$) \eqref{eq:converg_f*} is explicit in $N$, it is a very hard task to take the large population limit $N \to \infty$ (by using some basic algebraic manipulations) in order to simplify the expression for
\begin{equation}\label{eq:open}
\lim\limits_{N \to \infty} \lim\limits_{t \to \infty} \mathbb P[S_i(t) = n].
\end{equation}
We henceforth leave the computations/simplications of \eqref{eq:open} as a open problem. On the other hand, we will employ the mean-field approach in the next section and deduce the explicit equilibrium distribution (for the choice $f = f_*$).
\end{remark}

%In order to utilize Matlab to study the model and plot the long-term distribution of wealth, it is necessary to approximate $\varphi(N,M,R)$ and so we have the following result.

%\begin{lemma}\label{lem:approx}
%\begin{equation}\label{eq:approx}
%\begin{aligned}
%&\varphi(N,M,R) \approx \\
%&\sum_{b=0}^{N-1}\sum_{k=0}^{2(N-b)-1}\sum_{i=0}^{b-1}\frac{N!\left(\frac{M+N-b-1}{R}\right)^k\left(\frac{b}{R}\right)^i R^{2N-b-1}}{b!\,(N-b)!\,k!\,(2(N-b)-1-k)!\,i!\,(b-1-i)!\,(2N-b-1-k-i)}
%\end{aligned}
%\end{equation}
%\end{lemma}

\section{Formal mean-field limit}\label{sec:sec3}
\setcounter{equation}{0}

In this section, we carry out a formal mean-field argument of the stochastic agent-based dynamics \eqref{biased_exchange_with_debt} as the number of agents $N$ goes to infinity, following the procedure detailed in \cite{cao_uncovering_2022}.

The biased exchange model with collective debt limit can be written in terms of a system of stochastic differential equations, thanks to the framework set up in \cite{cao_derivation_2021,cao_uncovering_2022} for the study of the basic unbiased exchange model as well as some of its variants. Introducing $\{\mathrm{N}_t^{(i,j)}\}_{1\leq i,j\leq N}$, which represents a collection of independent Poisson processes with intensity $\frac{f(S_i)}{N}$, the evolution of each $S_i$ is given by:
\begin{equation}
  \label{eq:SDE}
  \dd S_i = -\sum \limits^{N}_{j=1} \underbrace{\left(1-\mathbbm{1}_{(-\infty,0]}(S_i)\cdot \delta_0(B_c)\right) \dd \mathrm{N}^{(i,j)}_{t}}_{\text{``$i$ gives to $j$''}} + \sum \limits^{N}_{j=1} \underbrace{\left(1-\mathbbm{1}_{(-\infty,0]}(S_j)\cdot \delta_0(B_c)\right) \dd \mathrm{N}^{(j,i)}_{t}}_{\text{``$j$ gives to $i$''}},
\end{equation}
where we use the notation $\delta_0(B_c) := \mathbbm{1}\{B_c = 0\}$. By the obvious symmetry, we can focus on the case when $i=1$ and notice that whenever $B_c \geq 1$, the SDE for $S_1$ simplifies to
\begin{equation}
\label{eq:SDE_phase1}
\dd S_1 = -\sum \limits^{N}_{j=1} \dd \NN^{(1,j)}_t + \sum \limits^{N}_{j=1} \dd \NN^{(j,1)}_t
\end{equation}
If we introduce
\begin{displaymath}
  \mathrm{\bf N}^1_t = \sum_{j=1}^N \mathrm{N}^{(1,j)}_t,\quad \mathrm{\bf M}^1_t = \sum_{j=1}^N \mathrm{N}^{(j,1)}_t,
\end{displaymath}
then the two Poisson processes $\mathrm{\bf N}^1_t$ and $\mathrm{\bf M}^1_t$ are of intensities $f(S_1)$ and $\frac{1}{N}\sum_{j=1}^N f(S_j)$, respectively.  Motivated by \eqref{eq:SDE_phase1}, we give the following definition of the limiting dynamics of $S_1(t)$ as $N \rightarrow \infty$ from the process point of view, providing that $B_c \geq 1$.

\begin{definition}[\textbf{Mean-field equation --- Phase \RN{1}}]
We define $\overbar{S}_1$ to be the compound Poisson process satisfying the following SDE:
\begin{equation}
\label{eq:SDE_phase1_limit}
\dd \overbar{S}_1 = -\dd \overbar{\mathrm{\bf N}}^1_t + \dd \overbar{\mathrm{\bf M}}^1_t,
\end{equation}
in which $\overbar{\mathrm{\bf N}}^1_t$ and $\overbar{\mathrm{\bf M}}^1_t$ are independent Poisson processes with intensities $f(\overbar{S}_1)$ and $\sum_{n \in \mathbb Z} f(n)\,p_n$, where ${\bf p}(t)=\left(\ldots,p_{-n}(t),\ldots,p_{-1}(t),p_0(t),p_1(t),\ldots,p_n(t),\ldots\right)$ denotes the law of the process $\overbar{S}_1(t)$, i.e. $p_n(t) = \mathbb P\left[\overbar{S}_1(t) = n\right]$.
\end{definition}

Its time evolution is given by the following dynamics:
\begin{equation}
    \label{eq:Q_1_a}
\frac{\dd}{\dd t} {\bf p}(t) = Q_1[{\bf p}(t)]
\end{equation}
with
\begin{equation}
  \label{eq:Q_1_b}
  Q_1[{\bf p}]_n := f(n+1)\,p_{n+1} + \left(\sum_{n \in \mathbb Z} f(n)\,p_n\right)\,p_{n-1} - f(n)\,p_n - \left(\sum_{n \in \mathbb Z} f(n)\,p_n\right)\,p_n ~~ \text{for } n \in \mathbb Z.
\end{equation}

The distribution ${\bf p}(t)$ naturally preserves its mass (since it is a probability mass function) and the mean value (as the total amount of money in the whole system is preserved), thus if we introduce the affine subspaces:
\begin{equation}\label{eq:probability_space}
\mathcal{S}_\mu = \{{\bf p} \mid \sum_{n \in \mathbb Z} p_n =1,~p_n \geq 0,~\sum_{n \in \mathbb Z} n\,p_n =\mu\}
\end{equation}
and \[\mathcal{S}^+_\mu = \{{\bf p} \in \mathcal{S}_\mu \mid p_n = 0~\textrm{for}~ n < 0\},\] then it is clear that the unique solution ${\bf p}(t)$ of \eqref{eq:Q_1_a}-\eqref{eq:Q_1_b} with ${\bf p}(0) \in \mathcal{S}^+_\mu$ satisfies ${\bf p}(t) \in \mathcal{S}_\mu$ for all $t > 0$. Moreover, we will introduce a class $\mathcal{G}$ of functions $f \colon \mathbb R \to (0,\infty)$ by setting
\begin{equation}\label{eq:class_G}
\mathcal{G} = \left\{f  ~~\Big|~~ \sum_{n\leq 0} f(n)\,p_n \cdot \sum_{n\geq 0} p_n - \left(\sum_{n \geq 1} f(n)\,p_n\right)\,\sum_{n\leq -1} p_n \geq 0,~~\forall~ {\bf p} \in \mathcal{S}_\mu.\right\}
\end{equation}
If we define the average amount of ``debt'' per agent as
\begin{equation}
  \label{eq:debt}
  D[{\bf p}] = -\sum_{n \leq -1} n\,p_n,
\end{equation}
then this quantity will be non-decreasing for all $f \in \mathcal{G}$ since
\begin{equation}\label{eq:debt_accumulation}
\begin{aligned}
\frac{\dd}{\dd t} D[{\bf p}(t)] &= \sum_{n\leq 0} f(n)\,p_n - \left(\sum_{n \in \mathbb Z} f(n)\,p_n\right)\,\sum_{n\leq -1} p_n \\
&= \sum_{n\leq 0} f(n)\,p_n \cdot \sum_{n\geq 0} p_n - \left(\sum_{n \geq 1} f(n)\,p_n\right)\,\sum_{n\leq -1} p_n \geq 0,
%&\geq \sum_{n\leq -1} f(n)\,p_n \cdot \sum_{n\geq 1} p_n - \left(\sum_{n \geq 1} f(n)\,p_n\right)\,\sum_{n\leq -1} p_n \geq 0,
\end{aligned}
\end{equation}
in which the last inequality follows directly from the assumption that $f \in \mathcal{G}$.

\begin{remark}
The class $\mathcal{G}$ of functions (at least) includes the following class of decreasing functions defined by
\begin{equation}\label{eq:class_D}
\mathcal{D} = \left\{f  \mid \textrm{$f$ is decreasing and $f(0) = 1$}.\right\}
\end{equation}
Indeed, if $f$ is decreasing and $f(0) = 1$, we have
\begin{align*}
&\sum_{n\leq 0} f(n)\,p_n \cdot \sum_{n\geq 0} p_n - \left(\sum_{n \geq 1} f(n)\,p_n\right)\,\sum_{n\leq -1} p_n \\
&\quad \geq \sum_{n\leq -1} f(n)\,p_n \cdot \sum_{n\geq 1} p_n - \left(\sum_{n \geq 1} f(n)\,p_n\right)\,\sum_{n\leq -1} p_n \geq 0.
\end{align*}
\end{remark}

From now on, we will always impose the condition $f \in \mathcal{G}$ on the function $f$ unless otherwise stated. Since the average amount of debt each agent can sustain in the underlying stochastic $N$-agents system is at most $\mu\,\nu$, we therefore terminate the evolution of \text{\bf Phase I} \eqref{eq:Q_1_a}-\eqref{eq:Q_1_b} until $t = t_*$, where
\begin{equation}
\label{eq:t_*}
t_* = \min\limits_{t\geq 0} \{D[{\bf p}(t)] =  \mu\,\nu\}.
\end{equation}

\begin{remark}
The primary reason for assuming that $f$ belongs to the class $\mathcal{G}$ lies the desire to ensure that $\frac{\dd}{\dd t} D[{\bf p}(t)] \geq 0$, which is a sufficient (but not necessarily) condition to guarantee that the \text{\bf Phase I} will terminate at a finite time $t_* < \infty$. It would be harder to show that the average amount of ``debt'' per agent $D[{\bf p}]$ will reach the threshold $\mu\,\nu$ in finite time when $f$ is not assumed to be a member of $\mathcal{G}$.
\end{remark}

At the level of the agent-based system, after the first time when there is no cash in the bank, i.e., when $t \geq t^{\textrm{stoc}}_* := \min\{\tau > 0 \mid B_c(\tau) = 0\}$, the analysis is much more involved so heuristic reasoning plays a major role in this manuscript. We notice that we have the following basic relations for all time $t \geq 0$:
\begin{equation}\label{eq:evolution_Bc}
B_c = B_* - \sum_{i=1}^N S^{-}_i,~~ \dd B_c = -\sum_{i=1}^N \dd S^{-}_i,
\end{equation}
and $B_c \geq 0$. Therefore, the evolution of $B_c$ is much faster than the evolution of each of the $S_i$'s, indicating that \eqref{eq:evolution_Bc} is really a ``fast'' dynamics compared to \eqref{eq:SDE}. These observations motivate the next definition:

\begin{definition}[\textbf{Mean-field equation --- Phase \RN{2}}]
\label{def:phase2}
We define $\widetilde{S}_1$ (for $t \geq t^{\textrm{stoc}}_*$) to be the compound Poisson process satisfying the following SDE:
\begin{equation}
\label{eq:SDE_phase2_limit}
\dd \widetilde{S}_1 = -\left(1 - \mathbbm{1}_{(-\infty,0]}(\widetilde{S}_1)\cdot Z\right)\dd \widetilde{\mathrm{\bf N}}^1_t + \left(1 - (1-Y)\cdot Z\right)\dd \widetilde{\mathrm{\bf M}}^1_t,
\end{equation}
in which $\widetilde{\mathrm{\bf N}}^1_t$ and $\widetilde{\mathrm{\bf M}}^1_t$ are independent Poisson processes with intensities $f(\widetilde{S}_1)$ and $\sum_{n \in \mathbb Z} f(n)\,\mathbb{P}\left[\widetilde{S}_1 = n\right]$. Moreover, $Y = Y(t) \sim \mathcal{B}\left(r(t)\right)$ and $Z = Z(t) \sim \mathcal{B}\left(q_0(t)\right)$ are independent Bernoulli random variables with parameters
\begin{equation}\label{eq:r}
  r = \mathbb{P}\left[\widetilde{S}_1 \geq 1\right]
\end{equation}
and \begin{equation}\label{eq:q0}
  q_0 = 1 - \frac{\mathbb{P}\left[\widetilde{S}_1 \leq -1\right]\cdot \sum_{n\geq 1} f(n)\,\mathbb{P}\left[\widetilde{S}_1=n \right]}{\sum_{n\leq 0} f(n)\,\mathbb{P}\left[\widetilde{S}_1=n \right]\cdot \mathbb{P}\left[\widetilde{S}_1 \geq 0\right]}.
\end{equation}
respectively.
\end{definition}

We again denote by ${\bf p}(t)=\left(\ldots,p_{-n}(t),\ldots,p_{-1}(t),p_0(t),p_1(t),\ldots,p_n(t),\ldots\right)$ the law of the process $\widetilde{S}_1(t)$ for $t \geq t_*$, i.e. $p_n(t) = \mathbb P\left[\widetilde{S}_1(t) = n\right]$ for $t \geq t_*$. Its time evolution is given by the following dynamics:
\begin{equation}
  \label{eq:Q_2_a}
\frac{\dd}{\dd t} {\bf p}(t) = Q_2[{\bf p}(t)]
\end{equation}
with
\begin{equation}
  \label{eq:Q_2_b}
Q_2[{\bf p}]_n = \left\{
    \begin{array}{ll}
      (1-q_0)\,\left[f(n+1)\,p_{n+1} - f(n)\,p_n\right] + \left[\tilde{r} + (\tilde{d} + p_0)\cdot(1-q_0)\right]\,(p_{n-1}-p_n), &\quad  n \leq -1,\\
      f(1)\,p_1 - (1-q_0)\,p_0 + \left[\tilde{r} + (\tilde{d} + p_0)\cdot(1-q_0)\right]\,(p_{-1}-p_0), & \quad n= 0, \\
      f(n+1)\,p_{n+1} - f(n)\,p_n + \left[\tilde{r} + (\tilde{d} + p_0)\cdot(1-q_0)\right]\,(p_{n-1}-p_n), & \quad n \geq 1,
    \end{array}
  \right.
\end{equation}
where
\begin{equation}
  \label{eq:Q_2_c}
  r = \sum_{n\geq 1} p_n \quad \text{and} \quad d = \sum_{n\leq -1} p_n \quad \text{and} \quad \tilde{r} = \sum_{n\geq 1} f(n)\,p_n \quad \text{and} \quad \tilde{d} = \sum_{n\leq -1} f(n)\,p_n.
\end{equation}
We remark here that $r$ and $d$ represent the proportion of ``rich'' and ``indebted'' agents, respectively.

\begin{remark}
We will illustrate the motivation behind Definition \ref{def:phase2}. Taking the large population limit $N \to \infty$, if we denote by ${\bf q}(t) = \left(q_0(t),q_1(t),\ldots\right)$ the law of $B_c(t)$, then the transition rates of $B_c(t)$ can be described by figure \ref{fig:evolution_bank} below.

\begin{figure}[!htb]
\centering
\includegraphics[width=.7\textwidth]{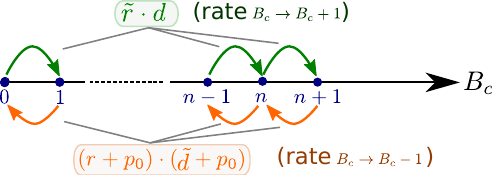}
\caption{Schematic illustration of the evolution of $B_c$ (amount of cash in the bank).}
\label{fig:evolution_bank}
\end{figure}

To be precise, the transition $B_c \to B_c + 1$ occurs at rate $\tilde{r}\cdot d$ when a ``rich'' agent is being picked ($S_i ≥ 1$) and gives a dollar to an ``indebted'' agent ($S_j<0$). Similarly, the transition $B_c \to B_c - 1$ occurs at rate $(\tilde{d} + p_0)\cdot (r+p_0)$ when an agent without dollar is being picked ($S_i≤0$) and gives a dollar to an agent without debt ($S_j≥0$) and moreover $B_c \geq 1$. As the evolution of $B_c$ is a ``fast'' dynamics (compared to the evolution of each agents' wealth), we assume that its distribution will converge to its ergodic invariant distribution within a time-scale that is negligible compared to the evolution of each of the $S_i$'s. Thus, from the following detailed balance equation at equilibrium
\[q_n\cdot \tilde{r}\cdot d = q_{n+1}\cdot (\tilde{d} + p_0) \cdot (r+p_0),\] one arrives at $q_n = \left(\frac{\tilde{r}\,d}{(r+p_0)\,(\tilde{d}+p_0)}\right)^n q_0$ for all $n\geq 0$. Due to identity $\sum_{n \geq 0} q_n = 1$, we deduce that
\[q_0 = 1 - \frac{\tilde{r}\,d}{(r+p_0)\,(\tilde{d}+p_0)},\] which coincides with the expression \eqref{eq:q0}. Meanwhile, the transition rates of $\widetilde{S}_1$ can be described by figure \ref{fig:evolution_tildeS}.

\begin{figure}[!htb]
  \centering
  \includegraphics[width=.97\textwidth]{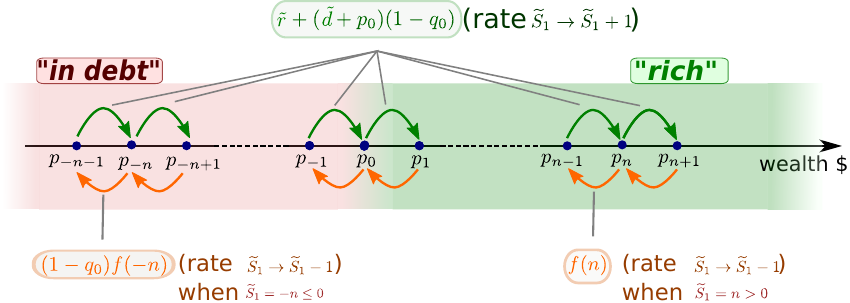}
\caption{Transition rates for the evolution of $\widetilde{S}_1$ in \text{\bf Phase II}.}
\label{fig:evolution_tildeS}
\end{figure}

This suggests that the evolution of ${\bf p}(t)$ should obey
\begin{equation}\label{eq:Qtilde}
\frac{\dd}{\dd t} {\bf p}(t) = \widetilde{Q}[{\bf p}(t)],
\end{equation}
with
\begin{equation}\label{eq:Qtilde_expression}
 \widetilde{Q}[{\bf p}]_n= \left\{
    \begin{array}{cclll}
      (1-q_0)\,(f(n+1)\,p_{n+1}-f(n)\,p_n)  &+& γ\,(p_{n-1}-p_n), & \quad n \leq -1,\\
      f(1)\,p_1 - (1-q_0)\,p_0  &+& γ\,(p_{-1}-p_0), & \quad n= 0, \\
      f(n+1)\,p_{n+1} - f(n)\,p_n  &+& γ(p_{n-1}-p_n), & \quad n \geq 1.
    \end{array}
  \right.
\end{equation}
with $γ=\tilde{r} + (\tilde{d} + p_0)\cdot(1-q_0)$. Thus, it is clear that the system \eqref{eq:Qtilde}-\eqref{eq:Qtilde_expression} coincides with \eqref{eq:Q_2_a}-\eqref{eq:Q_2_b}.
\end{remark}

In \text{\bf Phase II}, the average debt $-\sum_{n\leq -1} n\,p_n$ is preserved, therefore we can further restrict the affine space where the solution ${\bf p}(t)$ of \eqref{eq:Q_2_a}-\eqref{eq:Q_2_b} lives:
\begin{equation}\label{eq:SUV}
\mathcal{S}_{\mu,\nu} = \{{\bf p} \mid \sum_{n \in \mathbb Z} p_n =1,~p_n \geq 0,~\sum_{n \in \mathbb Z} n\,p_n =\mu,~-\sum_{n \leq 0} n\,p_n = \mu\,\nu\},
\end{equation}
it is straightforward to check that if ${\bf p}(t_*) \in \mathcal{S}_{\mu,\nu}$ then ${\bf p}(t) \in \mathcal{S}_{\mu,\nu}$ for all $t \geq t_*$.

We now turn to the search for the equilibrium distribution to \eqref{eq:Q_2_a}-\eqref{eq:Q_2_b}. Indeed, if we denote by ${\bf p}^* = \{p^*_n\}_{n \in \mathbb Z}$ the equilibrium solution to \eqref{eq:Q_2_a}-\eqref{eq:Q_2_b}, then for $n \geq 0$ we have
\[\frac{f(n+1)\,p^*_{n+1}}{p^*_n} = \frac{\sum_{k \geq 0} f(k+1)\,p^*_{k+1}}{\sum_{k \geq 0} p^*_k},\] leading to the conclusion that $\frac{f(n+1)\,p^*_{n+1}}{p^*_n}$ equals to the same constant for all $n \geq 0$. Similarly, we obtain that $\frac{f(n+1)\,p^*_{n+1}}{p_n}$ equals to the same constant for all $n \leq 0$. Therefore, for a general positive $f$ with $f(0) = 1$, the ``ansatz'' for the equilibrium distribution takes the following form
\begin{equation}\label{eq:p*n}
p^*_n = \left\{
    \begin{array}{ll}
     \frac{p^*_0}{f(n)\,\cdots\,f(1)\,f(0)}\,\beta^n_+, & \quad \text{if}~ n \geq 0,\\
     \frac{p^*_0}{f(n)\,\cdots\,f(-1)\,f(0)}\,\beta^{-n}_-, & \quad \text{if}~ n \leq 0.
    \end{array}
  \right.
\end{equation}
in which $\beta_{\pm} \in (0,1)$ and $p^*_0$ will be determined thought the following three constraints:
\begin{equation}\label{eq:constraints}
r^* + p^*_0+ d^* = 1, \quad \sum_{n \in \mathbb Z} n\,p^*_n = \mu, \quad \text{and} \quad -\sum_{n\leq -1} n\,p^*_n = \mu\,\nu.
\end{equation}

In particular, if we take $f = f_*$ where $f_*$ is given in Definition \ref{def_specific_f}, so that $f(n) \equiv 1$ for $n \leq 1$ and $f(n) = \frac{n-1}{n}$ for $n \geq 2$, then \eqref{eq:p*n} simplifies to
\begin{equation}\label{eq:p*n_example2}
p^*_n = \left\{
    \begin{array}{ll}
     n\,p^*_0\,\beta^n_+, & \quad \text{if}~ n \geq 1,\\
     p^*_0, &\quad \text{if}~ n = 0,\\
     \,p^*_0\,\beta^{-n}_-, & \quad \text{if}~ n \leq -1.
    \end{array}
  \right.
\end{equation}
Moreover, it is straightforward to check that $f_* \in \mathcal{G}$. Taking into account of the constraints \eqref{eq:constraints}, we obtain
\begin{subequations}\label{eq:3by3}
\begin{align}
p^*_0\,\left[1 + \frac{\beta_+}{(1-\beta_+)^2} + \frac{\beta_-}{1 - \beta_-} \right] & = 1,\label{eq:e1} \\
p^*_0\,\left[\frac{\beta^2_+ + \beta_+}{(1-\beta_+)^3} - \frac{\beta_-}{(1-\beta_-)^2}\right] &= \mu, \label{eq:e2}\\
p^*_0\,\frac{\beta_-}{(1-\beta_-)^2} &= \mu\,\nu \label{eq:e3}.
\end{align}
\end{subequations}
To solve the $3 \times 3$ system of nonlinear equations \eqref{eq:3by3}, we proceed as follows. Substituting \eqref{eq:e3} into \eqref{eq:e2} yields that
\begin{equation}\label{eq:e4}
\frac{1}{p^*_0} = \frac{\beta^2_+ + \beta_+}{(1-\beta_+)^3\,\mu\,(1+\nu)}.
\end{equation}
Also, we can deduce from \eqref{eq:e1} and \eqref{eq:e3} that
\begin{equation}\label{eq:e5}
\frac{\beta_-}{1 - \beta_-} = \frac{1}{p^*_0} - 1 - \frac{\beta_+}{(1-\beta_+)^2},
\end{equation}
and \begin{equation}\label{eq:e6}
\frac{\beta_-}{(1-\beta_-)^2} = \frac{1}{p^*_0}\,\mu\,\nu.
\end{equation}
Since $\frac{\beta_-}{(1-\beta_-)^2} = \frac{\beta_-}{1-\beta_-}\,\left(1 + \frac{\beta_-}{1-\beta_-}\right)$, we have
\begin{equation}\label{eq:e7}
\frac{1}{p^*_0}\,\mu\,\nu = \left(\frac{1}{p^*_0} - 1 - \frac{\beta_+}{(1-\beta_+)^2}\right)\left(\frac{1}{p^*_0} - \frac{\beta_+}{(1-\beta_+)^2}\right).
\end{equation}
Inserting \eqref{eq:e4} into \eqref{eq:e7} gives
\begin{equation*}
\frac{(\beta_+ + 1)\,\nu}{1+\nu}\,(1-\beta_+)^3 = \left(\frac{\beta^2_+ + \beta_+}{\mu\,(1+\nu)} - \beta_+\,(1-\beta_+) - (1 - \beta_+)^3\right)\left(\frac{\beta_+ + 1}{\mu\,(1+\nu)} - (1-\beta_+)\right),
\end{equation*}
which is equivalent to the following quartic equation (with only one unknown):
\begin{equation}\label{eq:quartic_eq}
c_4\,\beta^4_+ + c_3\,\beta^3_+ + c_2\,\beta^2_+ + c_1\,\beta_+ + c_0 = 0,
\end{equation}
where \[c_0 = 1 - \frac{\nu}{1+\nu} - \frac{1}{\mu\,(1+\nu)},~~c_1 = \frac{1}{\mu^2\,(1+\nu)^2} + \frac{2\,\nu}{1+\nu} - 3,\] \[c_2 = \frac{2}{\mu^2\,(1+\nu)^2} + 4,~~ c_3 = \frac{1}{\mu^2\,(1+\nu)^2} - \frac{2\,\nu}{1+\nu} - 3,~~\text{and}~~ c_4 = \frac{1}{\mu\,(1+\nu)} + \frac{\nu}{1+\nu} + 1.\]
It is natural to speculate that the system of nonlinear equations \eqref{eq:3by3} admits a unique solution for $(p^*_0, \beta_+, \beta_-) \in (0,1)^3$, for any prescribed $(\mu,\nu) \in \mathbb{N}^2_+$. Unfortunately we can not provide a rigorous proof of this natural conjecture. As a comprise, we are dedicated to solving the nonlinear system \eqref{eq:3by3} for specific choices of the pair $(\mu,\nu) \in \mathbb{N}^2_+$. For instance, when $\mu = \nu = 1$, one can find that $\beta_+ \approx 0.5852$, $p^*_0 = \frac{1}{21}\,\left(-11 + 29\,\beta_+ - 8\,\beta^2_+\right) \approx 0.15386$, and $\beta_- = \frac{1}{28}\,\left(25 - 15\,\beta_+ + 8\,\beta^2_+\right) \approx 0.6772$. The agent-based numerical simulation suggests that, as the number of agents and time go to infinity, the limiting distribution of money approaches the equilibrium distribution \eqref{eq:p*n_example2} as shown in figure \ref{simulation_agent_based_model}, in which we take $(\mu,\nu) = (1,1)$.

\begin{figure}[!h]
  \centering
  \includegraphics[scale=1]{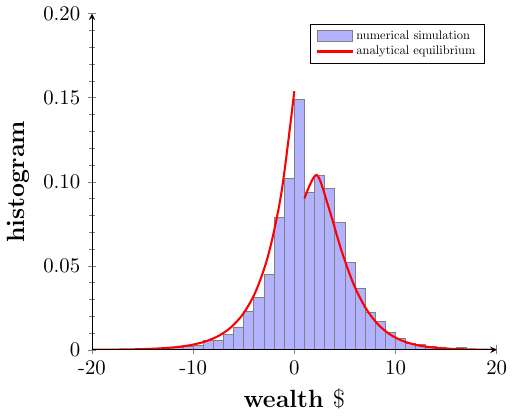}
  \caption{Simulation results for the $f$-biased exchange model with collective debt limit \eqref{biased_exchange_with_debt}, in which we take $f = f_*$ as given in Definition \ref{def_specific_f}, i.e., $f(n) = 1$ for $n \leq 1$ and $f(n) = (n-1)/n$ for $n \geq 2$. The blue histogram shows the distribution of money after $T = 500,000$ binary exchanges. Notice that the histogram is well-approximated by the distribution \eqref{eq:p*n_example2}. We used $N=10,000$ agents, each starting with $\mu = \$1$, and $\nu = \$1$.}
  \label{simulation_agent_based_model}
\end{figure}

To illustrate the convergence of ${\bf p}(t)$ toward the equilibrium ${\bf p}^*$, we run a simulation and plot the evolution of ${\bf p}(t)$ at various times (see figure \ref{fig:cv_equilibrium}-left) as well as the evolution of the $\ell^2$ distance between ${\bf p}(t)$ and ${\bf p}^*$ over time (see figure \ref{fig:cv_equilibrium}-right). We emphasize that the decay of the $\ell^2$ distance changes abruptly around $t_*$ which corresponds to the transition of the dynamics from phase I to phase II.

\begin{figure}[htbp]
  \centering
  \includegraphics[width=.97\textwidth]{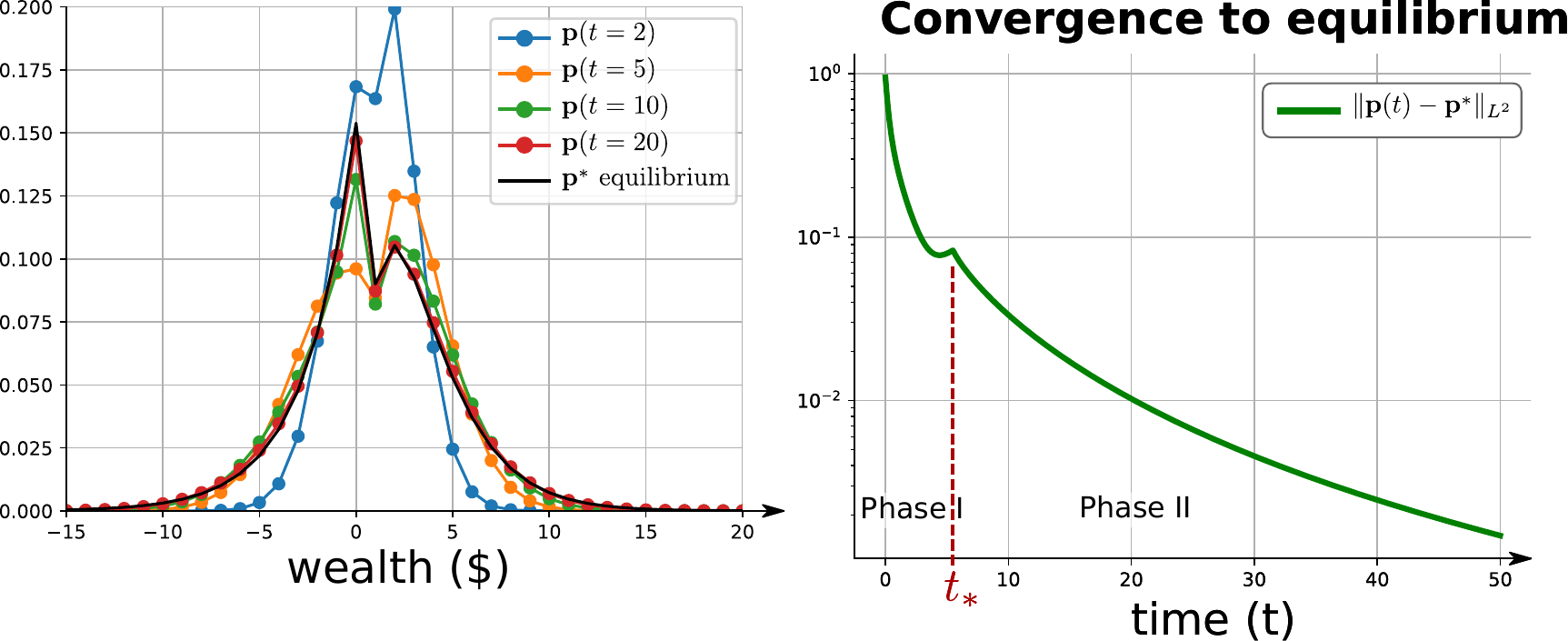}
  \caption{{\bf Left:} evolution of the distribution ${\bf p}(t)$ toward the equilibrium ${\bf p}^*$ \eqref{eq:p*n_example2}. The dynamics is in phase I till $t_*$. {\bf Right:} evolution of the $\ell^2$ distance between the solution ${\bf p}(t)$ and the equilibrium ${\bf p}^*$. In the simulation, we have set $\mu = \$1$ and $\nu = \$1$, along with the initial condition that $p_\mu(0)=1$ and $p_n(0) = 0$ for $n\neq \mu$.}
  \label{fig:cv_equilibrium}
\end{figure}

\begin{remark}
We now illustrate the main reasons for considering the ``strange-looking'' function $f_*$. Indeed, for the mean-field approach developed in \cite{cao_uncovering_2022} to be applicable in our generalized setting (where $f$ is no longer a fixed positive constant), we need to ensure that the quantity $D[{\bf p}]$ \eqref{eq:debt} to be non-decreasing during \textbf{Phase \RN{1}} so that we indeed have a two-phase dynamics. This consideration leads us to the restriction $f \in \mathcal{G}$. For example, if one wants to take the function $f$ to be $f(x) = 1$ for $-1\leq x \leq 1$ and $f(x) = \frac{|x|-1}{|x|}$ for $|x| > 1$, then an ansatz for the equilibrium distribution \eqref{eq:p*n} boils down to
\begin{equation}\label{eq:p*n_example}
p^*_n = \left\{
    \begin{array}{ll}
     n\,p^*_0\,\beta^n_+, & \quad \text{if}~ n \geq 1,\\
     p^*_0, &\quad \text{if}~ n = 0,\\
     -n\,p^*_0\,\beta^{-n}_-, & \quad \text{if}~ n \leq -1.
    \end{array}
  \right.
\end{equation}
However, with this choice of $f$ it is no longer clear whether $f \in \mathcal{G}$ or not.

Second of all, we also need to ensure that our choice of $f$ leads to the existence of a equilibrium distribution associated with \eqref{eq:Q_2_a}-\eqref{eq:Q_2_b} in the space \eqref{eq:SUV}. For instance, let us take $f(x) = \expo^{-\alpha\,x}$ for some $\alpha > 0$, which belongs to the class $\mathcal{D}$ \eqref{eq:class_D} but also approaches to infinity as $x \to -\infty$. In this scenario, the equilibrium ${\bf p}^*$ \eqref{eq:p*n} takes the form
\begin{equation}\label{eq:p*0}
p^*_n = \left\{
    \begin{array}{ll}
      \expo^{\alpha\,n\,(n+1)/2}\,p^*_0\,\beta^n_+, & \quad \text{if}~ n \geq 0,\\
      \expo^{\alpha\,n\,(n+1)/2}\,p^*_0\,\beta^{-n}_-, & \quad \text{if}~ n \leq 0,
    \end{array}
  \right.
\end{equation}
However, as $\alpha > 0$, ${\bf p}^*$ in the form of \eqref{eq:p*0} has no chance to be a probability mass function and our analytical framework breaks down. Lastly, we want to make sure that our choice $f$ allows an explicit computation of the equilibrium ${\bf p}^*$ at least for each fixed pair of the model parameters $(\mu,\nu) \in \mathbb{N}^2_+$, this consideration takes into account of the possibility of simulating the stochastic $N$-agents dynamics so that we can compare our analytical prediction with numerical observations. Putting together these concerns, we decide to pick $f$ to be $f_*$ as in Definition \ref{def_specific_f} in the majority part of the present paper.
\end{remark}

\section{Conclusion}
\setcounter{equation}{0}

In this manuscript, the so-called $f$-biased exchange model with collective debt limit is investigated, which serves as an extension of the unbiased exchange model with collective debt limit proposed in \cite{xi_required_2005} and revisited recently in \cite{cao_uncovering_2022,lanchier_rigorous_2018-1}. Although our generalized setting necessarily induces additional difficulty in the analysis using the framework developed in \cite{cao_uncovering_2022}, we are delighted to find that our analytical prediction of the asymptotic wealth distribution among agents as $N \to \infty$ and $t \to \infty$ based on our two-phase dynamics \eqref{eq:evolution_p} fits well with the numerical experiments, at least for a specific choice of the nonlinearity $f$.

To the best our of knowledge, rigorous treatment of econophysics models involving bank and debt is relatively rare in the literature and we refer to \cite{torregrossa_wealth_2017} for some rigorous analysis on certain econophysics models at the PDE level (whence completely bypassing the need to investigate the large population limit at the agent-based level). Our work also leaves many important open questions to be addressed in subsequent studies. For instance, is it possible to derive the two-phase dynamics \eqref{eq:evolution_p} rigorously~? Will it be possible to extend the framework developed in \cite{cao_uncovering_2022} so that analysis can be carried for broader class of functions $f$ (especially for whose functions $f \notin \mathcal{G}$)~? For those $f$ belonging to the class $\mathcal{G}$ such that ${\bf p}^*$ indeed defines a probability distribution, can we show rigorously that the solution of the \textbf{Phase \RN{2}} ODE system \eqref{eq:Q_2_a}-\eqref{eq:Q_2_b} converges to its equilibrium distribution ${\bf p}^*$~? We also emphasize that a rigorous study, based on the relative entropy, is possible when $f$ is assumed to a fixed positive constant \cite{cao_uncovering_2022}. Finally, for the Markov chain approach taken in section \ref{sec:sec2}, will it be possible to simplify the expression \eqref{eq:converg_f*} in Corollary \ref{corollary:convegence_f*} so that a closed and explicit form of the probability \eqref{eq:open} can be discovered, perhaps by resorting to advanced combinatorics tools or approximation theories~? We believe that answers to many of these problems deserve separate and thorough investigations on their own.\\

\noindent {\bf Acknowledgement} We highly appreciate the help from S\'ebastien Motsch on the generation of figure \ref{fig:cv_equilibrium}.

\end{document}